\documentclass[12pt,letterpaper]{amsart}

\usepackage{epsfig,amsmath,amsthm,amssymb,graphicx}
\usepackage[top=1in, bottom=1in, left=1in, right=1in]{geometry}
\usepackage{color}

\usepackage[dvipsnames]{xcolor}

\usepackage{faktor}

\usepackage{tikz}
\usetikzlibrary{matrix, arrows, cd}
\usepackage{pb-diagram}

\definecolor{myOlive}{rgb}{.29,.28,.16}
\definecolor{myRed}{rgb}{.78,0,0}

\makeatletter
\newtheorem*{rep@theorem}{\rep@title}
\newcommand{\newreptheorem}[2]{%
\newenvironment{rep#1}[1]{%
 \def\rep@title{#2 \ref{##1}}%
 \begin{rep@theorem}}%
 {\end{rep@theorem}}}
\makeatother
	
\newtheorem{proposition}{Proposition}[section]
\newtheorem{theorem}[proposition]{Theorem}

\newtheorem*{theorem*}{Theorem}
\newtheorem*{proposition*}{Proposition}
\newtheorem*{lemma*}{Lemma}
\newtheorem*{corollary*}{Corollary}

\theoremstyle{definition}

\newtheorem{question}[proposition]{Question}

\newreptheorem{theorem}{Theorem}
\newreptheorem{proposition}{Proposition}

\theoremstyle{remark}
\newtheorem{remark}[proposition]{Remark}

\newcommand{\bdry}{\partial}

\newcommand{\Z}{\mathbb{Z}}

\newcommand{\C}{\mathcal{C}}

\renewcommand{\int}{\operatorname{Int}}

\newcommand{\Prod}{\displaystyle \prod  }

\begin{document}
\title[Concordance and crossing changes]{Concordance, crossing changes, and knots in homology spheres}

\author{Christopher W.\ Davis}
\address{Department of Mathematics, University of Wisconsin--Eau Claire}
\email{daviscw@uwec.edu}
\urladdr{www.uwec.edu/daviscw}
\date{\today}


\begin{abstract}Any knot in $S^3$ may be reduced to a slice knot by  crossing changes.  Indeed, this slice knot can be taken to be the unknot. In this paper we study the question of when the same holds for knots in homology spheres.  We show that a knot in a homology sphere is nullhomotopic in a smooth homology ball if and only if that knot is smoothly concordant to a knot which is homotopic to a smoothly slice knot.  As a consequence, we prove that the equivalence relation on knots in homology spheres given by cobounding immersed annuli in a homology cobordism is generated by concordance in homology cobordisms together with homotopy in a homology sphere.  
\end{abstract}

\maketitle
\section{Introduction and statement of results}

Classically, a knot $K$ is an isotopy class of smooth embeddings of the circle, $S^1$, into the 3-sphere, $S^3$.   A knot is called \textbf{slice} if it forms the boundary of a smoothly embedded 2-disk $D$ in the 4-ball.  This disk $D$ is called a slice disk for $K$.   This notion was first considered by Fox and Milnor in 1957 in the study of singularities of surfaces in 4-manifolds \cite{FoMi57, FoMi}. The question of which knots are slice  is closely related to local obstructions arising in a surgery-theoretic attempt to classify 4-manifolds \cite{CF}.  Since then the question of what knots admit slice disks has been at the heart of the study of 4-manifold topology.  

While not every knot in $S^3$ is slice, every knot can be transformed into a slice knot by a finite sequence of crossing changes.  Indeed that slice knot can be taken to be the unknot.  In the case of knots in a non-simply connected homology sphere the situation is more subtle.  A knot representing a nontrivial class in the fundamental group cannot be reduced to the unknot by any sequence of crossing changes.  The main goal of this paper is to ask when a knot in a homology sphere can be homotoped to a new knot which bounds a smoothly embedded disk in homology ball.  We will consider the following question.

\begin{question}\label{quest:main}
Let $Y$ be a homology sphere and $K$ be a knot in $Y$.   Does there exist a homotopy transforming $K$ to a new knot $K'$ in $Y$ which bounds a smoothly embedded disk in a smooth homology ball bounded by $Y$?
\end{question}


If one allows for topological homology balls and locally flat embedded disks then the answer to this question is affirmative for all knots.  In \cite{AR99}, Austin and Rolfsen prove that any knot in a homology sphere admits a homotopy to a knot with Alexander polynomial 1.  By work of Freedman-Quinn \cite[Theorem 11.7B]{FQ} such a knot bounds a locally flat embedded disk in a contractible 4-manifold.  

In the smooth setting there are obstructions to Question~\ref{quest:main} having an affirmative answer.  First, not every homology sphere bounds a homology ball.   For example, $Y$ might be the Poincar\'e homology sphere or any other homology sphere with nonzero Rohlin invariant.  See \cite[Definition 5.7.16]{KirbyCalculus} for a brief discussion of the Rohlin invariant.  Secondly, if $K$ is homotopic in $Y$ to a knot which bounds an embedded disk in a homology ball, then $K$ is nullhomotopic in that homology ball.  By work of Daemi \cite[Remark 1.6]{Daemi2018} there exist a knot $K$ in a homology sphere $Y$ such that $Y$ bounds a homology ball and yet $K$ is not nullhomotopic in any homology ball bounded by $Y$.  Such a knot cannot be homotopic to a knot which bounds a smoothly embedded disk.

We give a name to the property of bounding a smoothly embedded disk in a homology ball.  Let $Y$ be a homology sphere and $K$ be a knot in $Y$.  We say that $(Y,K)$ is \textbf{homology slice} if there exists a smooth homology ball bounded by $Y$ in which $K$ bounds a smoothly embedded disk.  Thus, Question~\ref{quest:main} asks whether the homotopy class of $K$ in $Y$ contains a homology slice representative.

The notions of sliceness and homology sliceness extend  to equivalence relations.  Two knots $K$ and $J$ in $S^3$ are called \textbf{concordant} if $K\times\{1\}\subseteq S^3\times\{1\}$ and $J\times\{0\}\subseteq S^3\times\{0\}$ cobound a smoothly embedded annulus in $S^3\times[0,1]$.  We call knots $K$ and $J$  in homology spheres $Y$ and $X$ \textbf{homology concordant} if there exists a smooth homology cobordism from $Y$ to $X$ in which $K\subseteq Y$ and $J\subseteq X$ cobound a smoothly  embedded annulus.  Similar to the relationship between classical concordance and slice knots, the homology concordance class containing the unknot in $S^3$ is precisely the set of homology slice knots.  Importantly, homology concordance allows one to compare knots which do not lie in the same 3-manifold.  The quotient of the set of knots in homology spheres by homology concordance is the topic of study of \cite{Davis2018,DR2016,  HLL2018, Levine2016} amongst others.  

Our first main result says that every knot in the boundary of a contractible 4-manifold is homology concordant to a knot for which the answer to Question~\ref{quest:main} is affirmative.

\begin{theorem}\label{thm:smooth}
Let $Y$ be a homology sphere which bounds a smooth contractible 4-manifold.  Let $K$ be a knot in $Y$.  There exists a knot $K'$ in a homology sphere $Y'$ such that $(Y,K)$ is homology concordant to $(Y',K')$ and $K'$ is  homotopic in $Y'$ to a third knot $K''$ which bounds a smoothly embedded disk in a smooth contractible 4-manifold bounded by $Y'$.  
\end{theorem}

Said another way, if we let $\simeq_h^3$ and $\simeq_c$ denote  homotopy in a homology sphere and homology concordance respectively and we let $U$ denote the unknot in $S^3$ then Theorem~\ref{thm:smooth} concludes
$$
(Y,K)\simeq_c (Y',K')\simeq_h^3 (Y',K'')\simeq_c (S^3, U).
$$ 
Thus, the equivalence relation generated by homology concordance and homotopy equates every knot in the boundary of a contractible 4-manifold with the unknot.  

If one allows homology spheres which bound homology balls but not contractible 4-manifolds then there is an obstruction to a knot being related to a homology slice knot by a sequence of homotopies and homology concordances.  Indeed, by \cite[Remark 1.6]{Daemi2018} there exist knots in homology spheres which are not nullhomotopic in any homology ball.  Such a knot cannot be reduced to a smoothly slice knot by any sequence of homotopies and homology concordances.  The following theorem reveals that this is the only obstruction.  

\begin{theorem}\label{thm:smoothII}
Let $Y$ be a homology sphere which bounds a smooth homology ball $W$.  Let $K$ be a knot in $Y$ which is nullhomotopic in $W$.  There exists a knot $K'$ in a homology sphere $Y'$ such that $(Y,K)$ is homology concordant to $(Y',K')$ and $K'$ is homotopic in $Y$ to a third knot $K''$ which bounds a smoothly embedded disk in a smooth homology ball bounded by $Y'$.  
\end{theorem}

\begin{remark}
In \cite[Proposition 2.1]{StrleOwens2015} Strle and Owens prove a quantitative version of Theorem~\ref{thm:smooth} for knots in $S^3$.  They show that if a knot in $S^3$ bounds an immersed disk in the 4-ball which has $n$ self-intersections, then $K$ is concordant to another knot $K'$ (in $S^3$) which can be reduced to a slice knot after $n$ crossing changes.  While we will not do so in this document, by marrying their techniques with ours we expect the same can be proven of knots in homology spheres.  
\end{remark}

Theorems~\ref{thm:smooth} and \ref{thm:smoothII} both conclude that a knot is homology concordant to a knot for which the answer to Question~\ref{quest:main} is affirmative.  In the case of knots which are nullhomotopic in a homology ball admitting no 1-handles, work of Kojima \cite{Kojima79} implies that we can do away with the concordance, answering Question~\ref{quest:main} in the affirmative.  We prove the same result when the homology ball admits a handle structure with no 3-handles.  

\begin{theorem}\label{thm:smoothIII}
Let $Y$ be a homology sphere which bounds a smooth homology ball $W$ which has a handle structure with no 3-handles.  Let $K$ be a knot in $Y$ which is nullhomotopic in $W$.  Then $K$ is homotopic in $Y$ to a knot $K'$ which bounds a smoothly embedded disk in $W$.  
\end{theorem}

As an application consider the knot in a homology sphere $(Y,K)$ of Figure~\ref{fig:LevineExample}.  By \cite[Theorem 1.1]{Levine2016} this knot does not bound any piecewise linear embedded disk in any smooth homology ball.  As a consequence $(Y,K)$ is not homology concordant to any knot in $S^3$.  However, notice that $Y$ bounds a contractible 4-manifold which has one 0-handle, one 1-handle, one 2-handle, and most importantly no 3-handles.  Thus, by Theorem~\ref{thm:smoothIII}, $K$ is homotopic to another knot in $Y$ which is homology slice.  

\begin{figure}[h]
\begin{picture}(100,100)
\put(0,5){\includegraphics[height=.15\textheight]{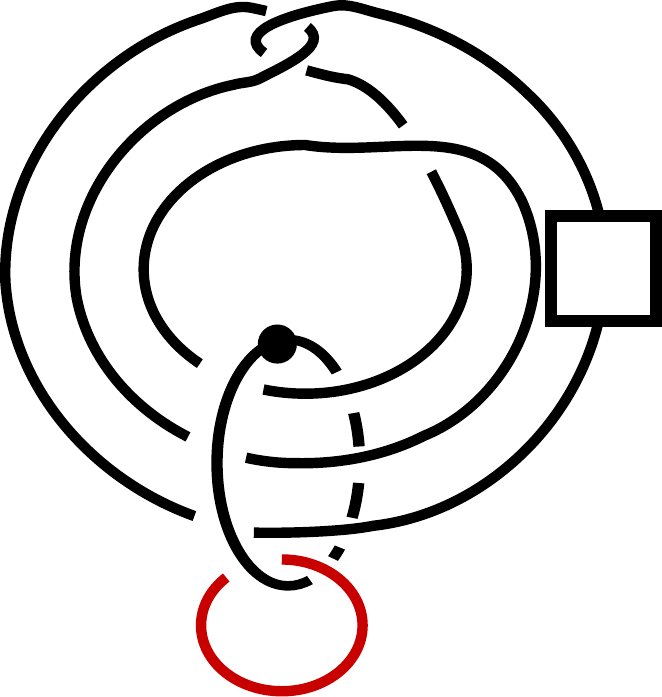}}
\put(81,61){$T$}
\put(49,0){\textcolor{myRed}{$K$}}
\put(5,90){$0$}
\end{picture}
\caption{A knot $K$ in the boundary of a contractible 4-manifold.  For a suitable knot $T$ this knot is not homology concordant to any knot in $S^3$ \cite[Theorem 1.1]{Levine2016}. }\label{fig:LevineExample}
\end{figure}
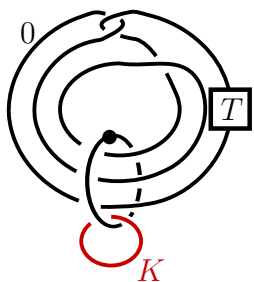 

There is an extension of the equivalence relations $\simeq_h^3$ and $\simeq_c$.  Given knots $K$ and $J$ in homology spheres $Y$ and $X$, if there exists a smooth homology cobordism from $Y$ to $X$ in which $K$ and $J$ are  homotopic then we write $(Y,K)\simeq_h^4(X,J)$.  Equivalently and more geometrically we say $(Y,K)\simeq_h^4(X,J)$ if there exists a smooth homology cobordism from $Y$ to $X$ in which $K$ and $J$ cobound an immersed annulus.   It is clear that if $(Y,K)$ and $(X,J)$ are related by a sequence of homotopies and homology concordances then they are related under $\simeq_h^4$.  The following theorem reveals the converse: if $(Y,K)\simeq_h^4(X,J)$ then $(Y,K)$ is related to $(X,J)$ by a sequence consisting of one homotopy and two homology concordances.  Thus, the equivalence relation $\simeq_h^4$ is generated by homotopy and homology concordance.

\begin{theorem}\label{thm:relative}
Let $(Y,K)$ and $(X,J)$ be knots in homology spheres.  If there exists a smooth homology cobordism from $Y$ to $X$ in which $K$ is  homotopic to $J$ then there exist knots $K'$ and $J'$ in some homology sphere $Z$ such that $$(Y,K)\simeq_c(Z,K')\simeq_h^3(Z,J')\simeq_c (X,J).$$  
\end{theorem}

\subsection*{Outline of paper}

In Section~\ref{sect:no 3-handles} we consider the case that a knot $K$ in a homology sphere $Y$ is nullhomotopic in a homology ball which has no 3-handles and prove Theorem~\ref{thm:smoothIII}.  In Section~\ref{sect:3-handles} we manipulate handle structures of homology balls in order to separate the 1- and 3-handles.  We go on to prove Theorems~\ref{thm:smooth} and \ref{thm:smoothIII}.  In Section~\ref{sect:relative} we take advantage of the group structure on $\C$ to prove Theorem~\ref{thm:relative}.  

\subsection*{Acknowledgments}

The author would like to thank MinHoon Kim, Arunima Ray, JungHwan Park, Paolo Acieto, Kent Orr, Aliakbar Daemi, and Adam Levine for helpful conversations at various stages of the development of the ideas leading to this paper.

\section{In the absence of 3-handles.}\label{sect:no 3-handles}

Throughout this paper, we make extensive use of handle decompositions of smooth manifolds.  A good reference is \cite{KirbyCalculus}.  
We begin with the following proposition  revealing that if $W$ is a 4-manifold, $Y\subseteq \bdry W$, and $(W,Y)$ has no relative 1-handles, then every homotopy class in $\ker(\pi_1(Y)\to \pi_1(W))$ is represented by a knot which bounds a smoothly embedded disk in $W$.  

\begin{proposition}\label{prop:no 3-handles}
Let $W$ be a smooth, connected, compact, orientable 4-manifold and $Y$ be a submanifold of $\bdry W$.  Suppose that $(W,Y)$ admits a handle structure with no 1-handles.  Let $\gamma\in \ker(\pi_1(Y)\to \pi_1(W))$.  There exists a knot $K$ in the homotopy class, $\gamma$, which bounds a smoothly embedded disk in $W$.   
\end{proposition}
\begin{proof}
 Let $\beta_1,\dots, \beta_m\subseteq Y$ be the cores of the attaching regions of the 2-handles of $(W, Y)$ regarded as framed knots.  Notice that if we take any collection of framed pushoffs of the various $\beta_i$, then the resulting link bounds a collection of disjoint smoothly embedded disks in $W$.  Indeed, these disks are pushoffs of the cores of the 2-handles.

Pick a base-point $q$ in $Y$.  
After a choice of basing arcs, the homotopy classes $[\beta_1],\dots, [\beta_m]$ normally generate $\ker(\pi_1(Y,q)\to \pi_1(W,q))$.  By assumption, $\gamma$ is in $\ker(\pi_1(Y,q)\to \pi_1(W,q))$.  We conclude that $\gamma$ is a product of conjugates of the $ \beta_i$:  $$\gamma = \Prod_{k=1}^n [c_k\beta_{i_k}^{\epsilon_{k}}c_k^{-1}]$$ for some choice of arcs $c_k\subseteq Y$ running from $q$ to a point on $\beta_k$ and $\epsilon_k = \pm1$.  Here $[c_k\beta_{i_k}^{\epsilon_{k}}c_k^{-1}]\in \pi_1(Y,q)$ is the homotopy class of the result of following $c_k$, then $\beta_{i_k}$ (or its reverse if $\epsilon_k=-1$), and finally the reverse of $c_k$.  After a small homotopy we assume that the various $c_k$ are embedded arcs, are disjoint from each other (except for the point at $q$), and have interiors disjoint from the various $\beta_i$.
 
As a consequence, we may construct a knot $K$ in the  the homotopy class $\gamma$ by starting with a single unknot based at $q$ and banding it with the pushoffs of $\beta_{i_k}$ (or the reverse $\beta_{i_k}^{-1}$) along bands around the $c_k$.  Finally, we build a smoothly embedded disk bounded by $K$ as follows.  For each $k=1,\dots, n$ let $\Delta_{k}$ be a pushoff of the core of the 2-handle attached along $\beta_{i_k}$ (or its orientation reverse if $\epsilon_k=-1$).  Observe that $\Delta_{1}, \dots, \Delta_{n}$ are disjoint smoothly embedded disks.  Band each $\Delta_{k}$ to a single disk bounded by the unknot centered at $q$ along the arc $c_k$.  The resulting surface is the required smoothly embedded disk.  
\end{proof}

Theorem~\ref{thm:smoothIII} amounts to a special case of Proposition~\ref{prop:no 3-handles}

\begin{reptheorem}{thm:smoothIII}
Let $Y$ be a homology sphere which bounds a smooth homology ball $W$ which has a handle structure with no 3-handles.  Let $K$ be a knot in $Y$ which is nullhomotopic in $W$.  Then $K$ is homotopic in $Y$ to a knot $K'$ which bounds a smoothly embedded disk in $W$.  
\end{reptheorem}
\begin{proof}
Let $Y$ be a smooth homology sphere which bounds a smooth homology ball $W$ which has a handle structure with no 3-handles.  Turning this handle structure upside-down $(W,Y)$ has a handle structure with no 1-handles.  Thus, Proposition~\ref{prop:no 3-handles} applies, so that since the homotopy class of $K$ is in $\ker(\pi_1(Y)\to \pi_1(W))$ we conclude that $K$ is homotopic in $Y$ to a knot which bounds a smoothly embedded disk in $W$.  
\end{proof}

\section{Separating 3-handles from 1-handles.}\label{sect:3-handles}

In the case of a homology ball $W$ which has 3-handles, we split $W$ into two pieces separating the 1-~and 3-handles.

\begin{proposition}\label{prop:decompose}
Let $W$ be a smooth homology ball with boundary $\bdry W = Y$.  Then there exists a smooth homology ball $W'$ with $\bdry W' = Y'$ and a smooth homology cobordism $V$ from $Y'$ to $Y$ such that 
\begin{enumerate}
\item $W'\cup_{Y'} V =W$
\item $W'$ has a handle structure with neither $3$- nor $4$-handles and so $(W',Y')$ has neither $0$- nor $1$-handles.
\item $(V,Y')$ has a handle structure with neither $0$- nor $1$-handles, so that $(V,Y)$ has neither $3$- nor $4$-handles.
\end{enumerate}
If  $W$ is a contractible 4-manifold then $W'$ can be arranged to also be contractible. 
\end{proposition}
\begin{proof}
Start by picking a handle decomposition of $W$ consisting of a single 0-handle, $n$ 1-handles, $m$ 2-handles, some number of 3-handles and no 4-handles. Let $W^{(d)}$ denote the union of all handles of dimension at most $d$.   Let $\alpha_1,\dots, \alpha_n\subseteq \bdry W^{(1)}$ be curves dual to the  belt spheres of the added 1-handles.  The homology classes $[\alpha_1]\dots[\alpha_n]$ give a basis for $H_1(W^{(1)})\cong\Z^n$.  Let $\beta_1,\dots, \beta_m\subseteq \bdry W^{(1)}$ be the cores of the attaching regions of the added 2-handles.  We express the homology classes $[\beta_1],\dots, [\beta_m] \in H_1(W^{(1)})$ as vectors in terms of the basis $[\alpha_1],\dots, [\alpha_n]$ in order to get an $(m\times n)$ presentation matrix, $P$, for $H_1(W)=0$.  See for example \cite[Section 4.2]{KirbyCalculus}.   

As $P$ presents the trivial group, a sequence of row-moves reduces $P$ to an $(m\times n)$ matrix with $1$'s on the main diagonal and $0$'s elsewhere.  These row-moves in turn can be realized by reordering and reversing some $\beta_i$'s and by sliding some $\beta_i$'s over other $\beta_j$'s.  See for example \cite[Section 5.1]{KirbyCalculus}.  Thus, after performing these  moves, we may assume that in $H_1(W^{(1)})$, $[\beta_i]=[\alpha_i]$ for all $i\le n$ and $[\beta_i]=0$ for all $i>n$.  

Let $W'$ be the result of adding to $W^{(1)}$ 2-handles along $\beta_1,\dots, \beta_n$.  Since $[\beta_i] =[\alpha_i]$ in $H_1(W^{(1)})$ for $i=1,\dots n$ we see that $H_1(W')=H_2(W')=0$.  Because $W'$ has neither 3- nor $4$-handles, $H_3(W') = H_4(W')=0$.  Thus, $W'$ is a homology ball.  
 Let $Y'=\bdry W'$ and $V$ be given by starting with $Y'\times[0,1]$, adding 2-handles along $\beta_{n+1},\dots, \beta_m$ and finally adding all 3-handles.  Then $W=W'\cup_{Y'} V$.  Because $W$ and $W'$ are both homology balls, a Mayer-Veitoris argument reveals that $V$ is a homology cobordism.  By its very construction, $(V,Y')$ has only 2- and 3-handles.  This gives the desired result when $W$ is a homology ball. 

 Now suppose that $W$ is a contractible 4-manifold.  Pick a basepoint $q\in \bdry W^{(1)}$.  After a choice of  basing arcs, we see that  $\pi_1(\bdry W^{(1)}, q)\cong \pi_1(W^{(1)}, q)$ is the free group on the set of homotopy classes $[\alpha_1],\dots, [\alpha_n]$.  Moreover, as $W$ is simply connected, 
 $$\pi_1(\bdry W^{(1)}) = \ker\left(\pi_1(\bdry W^{(1)},q)\to \pi_1(W,q) = \{0\}\right)$$ 
 is normally generated by the homotopy classes of the cores of the attaching regions of the 2-handles $[\beta_1],\dots, [\beta_m]$, after a choice of basing arcs.  Thus, each $\alpha_i$ is homotopic in $\bdry W^{(1)}$ to a product of conjugates of the attaching regions:
 $$
 [\alpha_i] = \Prod_{k=1}^{\ell_i} [c_{i,k} \beta_{j_{i,k}}^{\epsilon_{i,k}} c_{i,k}^{-1}]
 $$
 for some choices of arcs $c_{i,k}$ running from $q$ to a point on $\beta_{j_{i,k}}$ and $\epsilon_{i,k} = \pm1$.  After a small homotopy we may assume that the various $c_{i,k}$ are all embedded, have disjoint interiors, and have interiors disjoint from the various $\beta_j$.  Similar to the proof of Proposition~\ref{prop:no 3-handles} we build a knot $A_i$ in $\bdry W^{(1)}$ representing the homotopy class $[\alpha_i]$ by starting with a small unknot near $q$ and banding with framed pushoffs of the various $\beta_{j_{i,k}}$ (or their reverses if $\epsilon_{i,k}=-1$) along bands centered on the $c_{i,k}$.  Let $A_1,\dots, A_n$ be the resulting knots.  By construction $[A_i] = [\alpha_i]$ in $\pi_1(\bdry W^{(1)}, q)$.  
 
Since $A_i$ is constructed by starting with the unknot and banding it with the $\beta_{j_{i,k}}$, we may slide each $A_i$ over the various $\beta_{j_{i,k}}$ to reduce the link $A_1\cup \dots\cup A_n$ to the unlink in $\bdry W^{(2)}$.  Thus, the various $A_i$ bound disjoint disks in $\bdry W^{(2)}$.  Restrict a framing of these disks to a framing on $A_i$.  Form a new 4-manifold by starting with $W^{(2)}$ and adding 2-handles along the knots $A_1,\dots, A_n$.  Because the $A_i$ bound disjoint disks in $\bdry W^{(2)}$ the boundary of this new 4-manifold is the connected sum of $\bdry W^{(2)}$ with $n$ copies of $S^1\times S^2$.  Add 3-handles along these new non-separating 2-spheres.  The 2-handles added to $A_i$ together with these 3-handles form cancelling pairs \cite[Proposition 4.2.9]{KirbyCalculus} so that the 4-manifold resulting from adding these 2- and 3-handles to $W^{(2)}$ is diffeomorphic to $W^{(2)}$.  Construct a 4-manifold diffeomorphic to $W$ by adding the remaining 3-handles.  

Thus, we may assume that $W$ has a handle structure with 2-handles attached along framed curves $A_1, \dots, A_n, \beta_1,
\dots, \beta_m\subseteq \bdry W^{(1)}$ where $[A_1], \dots, [A_n]$ give a free basis for $\pi_1(W^{(1)})$.  Let $W'$ be given by adding to $W^{(1)}$ only the 2-handles attached along $A_1,\dots, A_n$.  It  now follows that $W'$ is a simply connected homology ball.  Therefore, $W$ is contractible by a standard application of the Hurewicz homomorphism and the Whitehead theorem.  See for example \cite[Corollary 6.70]{DavisKirk01}.  Let $Y'=\bdry W'$ and $V$ be given by starting with $Y'\times[0,1]$,  adding 2-handles along $b_1,\dots, b_m$ and then  3-handles.  This gives the desired result.
\end{proof}

Finally, we prove Theorems \ref{thm:smooth} and  \ref{thm:smoothII}.

\begin{reptheorem}{thm:smooth}
Let $Y$ be a homology sphere which bounds a smooth contractible 4-manifold.  Let $K$ be a knot in $Y$.  There exists a knot $K'$ in a homology sphere $Y'$ such that $(Y,K)$ is homology concordant to $(Y',K')$ and $K'$ is homotopic in $Y'$ to a third knot $K''$ which bounds a smoothly embedded disk in a smooth contractible 4-manifold bounded by $Y'$.  
\end{reptheorem}

\begin{proof}
Let $Y$ be the boundary of a contractible smooth 4-manifold $W$.  Let $K$ be a knot in $Y$.  
%
  Appeal to Proposition~\ref{prop:decompose} to decompose $W$ as $W'\cup_{Y'} V$ where $W'$ is a contractible 4-manifold with boundary $Y'$, $V$ is a homology cobordism from $Y$ to $Y'$,  $(V,Y)$ has only relative 1-~and 2-handles, and $(W',Y')$ has no 1-handles.  As $(V,Y)$ has only 1-~and 2-handles, we may realize $V$ by starting with $Y\times[0,1]$ and adding $1$- and $2$-handles on $Y\times\{1\}$.  We isotope $K$ slightly to make it disjoint from the attaching regions for these handles.  The image of $K\times[0,1]$ in $Y\times[0,1]\subseteq V$  gives a homology concordance from $K$ to some knot $K'$ in $Y'$.  Now, $\pi_1(W')$ is trivial so that $K'$ is nullhomotopic in $W'$.  Since $W'$ has no 3-handles, Theorem~\ref{thm:smoothIII} applies and we see that $K'$ is homotopic to some other knot $K''$ in $Y'$ which bounds a smoothly embedded disk in $W'$.
\end{proof}

\begin{reptheorem}{thm:smoothII}
Let $Y$ be a  homology sphere which bounds a homology ball $W$.  Let $K$ be a knot in $Y$ which is nullhomotopic in $W$.  There exists a knot $K'$ in a homology sphere $Y'$ such that $(Y,K)$ is homology concordant to $(Y',K')$ and $K'$ is homotopic in $Y$ to a third knot $K''$ which bounds a smoothly embedded disk in a smooth homology ball bounded by $Y'$.  
\end{reptheorem}

\begin{proof}
Let $Y$ be the boundary of a smooth homology ball $W$.  
%
  The proof begins in the same manner as the proof of Theorem~\ref{thm:smooth}.  Appeal to Proposition~\ref{prop:decompose} to decompose $W$ as $W'\cup_{Y'} V$ where $W'$ is a homology ball with boundary $Y'$, $V$ is a homology cobordism from $Y$ to $Y'$,  $(V,Y)$ has only relative 1-~and 2-handles, and $(W',Y')$ has no 1-handles.  As $(V,Y)$ has only 1-~and 2-handles, we realize $V$ by starting with $Y\times[0,1]$ and adding $1$- and $2$-handles to $Y\times\{1\}$.  We may isotope $K$ slightly to get $K$ disjoint from the attaching regions.  The image of $K\times[0,1]$ in $Y\times[0,1]\subseteq V$  gives a concordance from $K$ to some knot $K'_0$ in $Y'$.   From here the proof differs from that of Theorem~\ref{thm:smooth}.  As $W'$ is not simply connected we cannot conclude that $K'_0$ is  nullhomotopic in $W'$.  

Let $K'_+$ be a pushoff of $K_0'$. Pick a basepoint $q\in Y'$ which lies on $K'_+$.  Since $K$ is nullhomotopic in $W$ by assumption, $K'_+$ is nullhomotopic in $W$ as well.  Therefore, the homotopy class $[K'_+]$ lies in $\ker (\pi_1(V,q)\to \pi_1(W,q))$.  Recall that  $(W',Y')$ has only 2-~and 3-handles.  After making a choice of basing arcs, the homotopy classes of the attaching regions for these 2-handles, $\beta_1,\dots, \beta_n\subseteq Y'$, normally generate $\ker (\pi_1(V,q)\to \pi_1(W,q))$.  Thus
$[K'_+]$ is equal in $\pi_1(V,q)$ to a product of conjugates of these $\beta_i$. More precisely, in $\pi_1(V, q)$, 
$
[K'_+] =
\Prod_{k=1}^m [c_k\beta_{i_k}^{\epsilon_k} c_k^{-1}]
$
for some choices of embedded curves $c_k\subseteq V$ running from $q$ to a point on $\beta_{i_k}$.  Since $(V,Y')$ has no relative 1-handles, $\pi_1(Y',q)\to \pi_1(V,q)$ is onto and we may assume that each $c_k$ is embedded in $Y'$.  With this assumption, we have $\Prod_{k=1}^m [c_k\beta_{i_k}^{\epsilon_k} c_k^{-1}]\in \pi_1(Y', q)$ is a product of conjugates of the attaching regions $\beta_1,\dots, \beta_n$.  Thus, $\Prod_{k=1}^m [c_k\beta_{i_k}^{\epsilon_k} c_k^{-1}]$  is in $\ker(\pi_1(Y',q)\to \pi_1(W',q))$.  Theorem~\ref{thm:smoothIII} concludes that $\Prod_{k=1}^m [c_k\beta_{i_k}^{\epsilon_k} c_k^{-1}] \in \pi_1(Y',q)$ is represented by a knot $K''$ in $Y'$ which bounds a smoothly embedded disk in $W'$.  

Summarizing the proof so far, there is a smoothly embedded annulus $C$ in $V$ bounded by $K\subseteq Y$ and $K_0'\subseteq Y'$.  In turn $K_0'$ is homotopic in $V$ to another knot $K''\subseteq Y'$ which bounds a smoothly embedded disk in $W'$.  It remains to alter the annulus $C$ and the knot $K_0'$ to arrange that the homotopy from $K_0'$ to $K''$ lies in $Y'$.  


  Let $\nu(K)\cong K\times D^2$ be a product neighborhood of $K$ in $Y$, $\nu(C)$ be the image of $\nu(K)\times[0,1]$ in $ Y\times[0,1]\subseteq V$, and $\nu(K') = \nu(C)\cap Y'$.  Choosing these tubular neighborhoods small enough, $K'_+$ is disjoint from $\nu(K')$.  As $[K'_+]=[K'']$ in $\pi_1(V,q)$ it follows that the product $[K'{_+}]^{-1}[K'']$ is nullhomotopic in $V$.  Thus, $[K'{_+}]^{-1}[K'']\in \ker(\pi_1(V-\nu(C), q)\to \pi_1(V, q))$, which is normally generated by the homotopy class of a meridian $m(K'_0)\subseteq Y'$ of $K'_0$ based at $q$.   
Thus, we conclude that there is a product of conjugates of meridians of $K'_0$, $\delta = \Prod_{k} [d_k m(K'_0)^{\epsilon_k} d_k^{-1}]$ with each $d_k$ and embedded arc in $V-\nu(C)$, such that 
$[K'{_+}]^{-1}[K'']\delta^{-1}$
is nullhomotopic in $V-\nu(C)$.  

As $\nu(C)$ is the image of $\nu(K)\times[0,1]$ in $Y\times[0,1]\subseteq V$ and $(V,Y)$ has only 1-~and 2-handles, it follows that $(V-\nu(C),Y-\nu(K))$ has only 1-~and 2-handles.  Turning this handle structure upside-down, $(V-\nu(C), Y'-\nu(K'_0))$ has only relative 2-~and 3-handles.  In particular, $\pi_1(Y-\nu(K_0'))\to \pi_1(V-\nu(C))$ is onto.  Thus, we may assume that each $d_k$ lies in $Y'-\nu(K_0')$.  As a consequence, $\delta\in \pi_1(Y'-\nu(K_0'), q)$.  

 We now have that $[K'{_+}]^{-1} [K'']\delta^{-1}$ is in $\ker(\pi_1(Y'-\nu(K_0')) \to \pi_1(V-\nu(C)))$.   Again making use of the fact that $(V-\nu(C), Y'-\nu(K'_0))$ has no 1-handles, Proposition~\ref{prop:no 3-handles} implies that the homotopy class $[K'{_+}]^{-1} [K'']\delta^{-1}\in \pi_1(Y'-\nu(K_0'))$ contains a knot $J$ which bounds a smoothly embedded disk $D_J$ in $V-\nu(C)$.  The disk $D_J\subseteq V$ is disjoint from $C$.  We band $C$ to $D_J$ along an embedded arc in $Y'$ and see a homology concordance from $K$ to a new knot 
$K'$ with $[K'] = [K'_+][J]$ in $\pi_1(Y',q)$.  Expanding in $\pi_1(Y',q)$, 
$$
[K'] = [K'_+][J] = [K'_+]\left([K'{_+}]^{-1} [K'']\delta^{-1}\right) = [K'']\delta^{-1}
$$
 As $\delta$ is a product of conjugates of meridians of $K_0'$, $\delta$ is nullhomotopic in $Y'$.  Thus, in $\pi_1(Y',q)$, $[K'] = [K'']$.  
 
 Summarizing, there is a homology cobordism $V$ from $Y$ to $Y'$  in which the knots $K$ and $K'$ cobound a smoothly embedded annulus.  The knot $K'\subseteq Y$ is homotopic in $Y$ to a third knot $K''$,  which bounds a smoothly embedded disk in the homology ball $W'$.  This completes the proof. \end{proof}
 
 \section{The equivalence relation generated by homotopy and homology concordance}\label{sect:relative}

Next we set our eyes on Theorem~\ref{thm:relative} which concludes that two knots in possibly different homology spheres cobound an immersed annulus in a homology cobordism if and only if they are related by a sequence of homology concordances and homotopies. Recall the following notations from the introduction:
\begin{itemize}
\item $(X,K)\simeq_h (X,J)$ means $K$ and $J$ cobound an immersed annulus in the homology sphere $X$.
\item $(X,K)\simeq_c (Y,J)$ means there exists a homology cobordism from $X$ to $Y$ in which $K$ and $J$ cobound a smoothly embedded annulus.
\item $(X,K)\simeq_h^4 (Y,J)$ means there exists a homology cobordism from $X$ to $Y$ in which $K$ and $J$ cobound an immersed annulus.  
\end{itemize}
First we check that all of these are well defined under connected sum of pairs.  Recall that the connected sum $(X,K)\#(Y,J) = (X\# Y, K\# J)$ is given by picking points $p\in K$ and $q\in J$, removing small neighborhoods $\nu(p)\subseteq X$ and $\nu(q)\subseteq Y$, of $p$ and $q$, so that $K-\nu(p)$ and $J-\nu(q)$ are properly embedded arcs from a point $p_+$ to a point $p_-$ and from $q_+$ to  $q_-$.  We form the connected sum $X\# Y$ by gluing these two spherical boundaries together along an orientation reversing diffeomorphism which sends $p_+$ to $q_-$ and $p_-$ to $q_+$.  The connected sum $K\# J$ is given by gluing together $K-\nu(p)$ and $J-\nu(q)$.

\begin{proposition}\label{prop: well defined}
Suppose that $(Y,K)$, $(Y',K')$, and $(X,J)$ are knots in homology spheres.  For any $\simeq\in\{\simeq_h, \simeq_c,\simeq_h^4\}$ if $(Y,K)\simeq (Y',K')$ then $(Y,K)\#(X,J)\simeq (Y',K')\#(X,J)$.  (In the case of $\simeq_h$ we assume $Y=Y'$.)
\end{proposition}
\begin{proof}
The proofs in the cases of $\simeq_c$ and $\simeq_h^4$ are nearly identical. Suppose that $(Y,K)$, $(Y',K')$, and $(X,J)$ are knots in homology spheres. Let $W$ be a homology cobordism from $Y$ to $Y'$ in which $A$ is an embedded (or immersed in the case of $\simeq_h^4$) annulus bounded by $K$ and $K'$.  Let $\alpha$ be an embedded curve in $A$ running from a point on $K$ to a point on $K'$.  If $A$ is immersed, then $\alpha$ is chosen to be disjoint from all self-intersection points of $A$.  Let $p$ be a point on $J$.  Glue together $W-\nu(\alpha)$ and $(X-\nu(p))\times[0,1]$ to get a homology cobordism $V$ from $Y\# X$ to $Y'\# X$.  Let $ A'$ be the result of gluing $A-\nu(\alpha)$ to $(J-\nu(p))\times[0,1]$ in this homology cobordism.  Then $A'$ is an embedded (or immersed if $A$ is immersed) annulus bounded by $K\#J$ and $K'\#J$. Thus, if $(Y,K)\simeq_c (Y',K')$ then $(Y\#X, K\#J)\simeq_c (Y'\#X, K'\#J)$ and if $(Y,K)\simeq_h^4 (Y',K')$ then $(Y\#X, K\#J)\simeq_h^4 (Y'\#X, K'\#J)$.

Suppose that $(Y,K)\simeq_h(Y',K')$ so that $Y=Y'$ and $K$ is homotopic to $K'$ in $Y$.   We may assume that the homotopy is constant on a small neighborhood of some point $p$ of $K$.  Using this as the point we use in the connected sum construction and using the constant homotopy on $J\subseteq Y$, we see that $K\#J$ and $K'\#J$ are homotopic in $Y\#X$, completing the proof.
\end{proof}

\begin{reptheorem}{thm:relative}
Let $Y$ and $X$ be homology spheres.  Let $K$ and $J$ be knots in $X$ and $Y$ respectively.  If there exists a smooth homology cobordism from $Y$ to $X$ in which $K$ is homotopic to $J$ then there exist knots $K'$ and $J'$ in some homology sphere $Z$ such that $$(Y,K)\simeq_c(Z,K')\simeq_h(Z,J')\simeq_c (X,J).$$  
\end{reptheorem}

\begin{proof}[Proof of Theorem~\ref{thm:relative}]
Suppose that $(X,K)$ and $(Y,J)$ are knots in homology spheres and that $(X,K)\simeq_h^4 (Y,J)$.  
%
By Proposition~\ref{prop: well defined}, $(X\#-Y,K\#-J)\simeq_h^4 (Y\#-Y,J\#-J)$.  Here $-Y$ indicates the orientation reverse of $Y$ and $-J$ the orientation reverse of $J$.  Just as for knots in $S^3$, $J\#-J$ bounds a smoothly embedded disk in a homology ball bounded by $Y\#-Y$.  Stack a homology cobordism $X\#-Y$ to $Y\#-Y$ on top of a homology ball bounded by $Y\#-Y$. This gives a homology ball bounded by $X\#-Y$.  In this homology cobordism stack a homotopy from $(X\#-Y,K\#-J)$ to $(Y\#-Y,J\#-J)$ on top of a slice disk for $J\#-J$ to see an immersed disk.  Thus, $K\#-J$ is nullhomotopic in a homology ball.   Theorem~\ref{thm:smoothII} now applies and concludes that there exists another homology sphere $Z_0$ and knots $L$ and $L'$ in $Z_0$ such that 
$$
(X\#-Y,K\#-J)\simeq_c(Z_0,L)\simeq_h(Z_0,L')\simeq_c(S^3,U)
$$ 
where $U$ is the unknot in $S^3$.  By Proposition~\ref{prop: well defined}, we may take the connected sum of each of these terms with $(Y,J)$ to get
$$
((X\#-Y)\#Y,(K\#-J)\#J)\simeq_c(Z_0\#Y,L\# J)\simeq_h(Z_0\#Y,L'\#J)\simeq_c(Y,J)
$$ 
This together with the observation that ${(X,K)\simeq_c(X\#(-Y\#Y),K\#(-J\#J))}$ completes the proof. 
\end{proof}

\bibliographystyle{plain}

\bibliography{biblio}

\end{document}